\documentclass[12pt,a4paper]{article}
\usepackage[centertags]{amsmath}
\usepackage{amsfonts}
\usepackage{amssymb}
\usepackage{amsthm}
\usepackage{color}
\usepackage{marginnote}
\usepackage{float}

\usepackage[all]{xy}
\usepackage{fullpage}
\usepackage{tikz}
\usepackage{verbatim}
\usetikzlibrary{arrows}

\newtheorem{coro}{Corollary}
\newtheorem{defi}{Definition}

\newtheorem{rem}{Remark}
\newtheorem{prop}{Proposition}
\newtheorem{lem}{Lemma}

\newcommand{\MTL}{\mathcal{MTL}}
\newcommand{\QMTL}{\mathit{lu}\mathcal{MTL}}
\newcommand{\SH}{\mathcal{SH}}

\newcommand{\Id}{\mathcal{I}}
\newcommand{\nId}{\mathcal{I}^{\ast}}
\newcommand{\LU}{\mathcal{LU}}

\usepackage{color}

\begin{document}

\title{Generalized Chain Products of MTL-chains}

\author{J. L. Castiglioni and W. J. Zuluaga Botero}


\maketitle

\begin{abstract}
\noindent
In this paper we present a different approach to ordinal sums of MTL-chains, as extensions of finite chains in the category of semihoops. In addition we prove in a very simple way that every finite locally unital MTL-chain can be decomposed as an ordinal sum of archimedean MTL-chains. Furthermore, we introduce Generalized Chain Products of MTL-chains and we show that ordinal sums of locally unital MTL-chains are particular cases of those.
\end{abstract}

\section*{Introduction}

In \cite{EG2001} Esteva and Godo proposed a new Hilbert-style calculus called Monoidal T-norm based Logic (MTL, for short) in order to find the fuzzy logic corresponding to the bigger class of all left-continuous t-norms. In \cite{JM2002} Jenei and Montagna proved that MTL was in fact, the weakest logic which is complete with respect to a semantics given by a class of t-norms and their residua. Since in MTL the contraction rule in general does not hold, such logic can be regarded not only as a fuzzy logic and as a many-valued logic, but also as a substructural logic. These results motivated to introducing a new class of algebras with an  equivalent algebraic semantics for MTL, the variety of MTL-algebras. MTL-algebras are essentially integral commutative residuated lattices with bottom satisfying the prelinearity equation:
$$
(x \to y) \vee (y \to x) \approx 1
$$

\noindent
In \cite{CZ2017} Castiglioni and Zuluaga characterized the class of finite MTL-chains which can be decomposed as an ordinal sum of archimedean MTL-chains. Such a class of finite MTL-algebras was called \emph{locally unital}. Nevertheless, the general problem of decompositions by ordinal sums of (arbitrary) MTL-chains still remains open. It is worth mentioning that the general problem was already solved for the subvariety of BL-chains (c.f \cite{AM2003},\cite{B2004}) by proving that (arbitrary) BL-chains can be decomposed as an ordinal sum of Wajsberg hoops.
\\

\noindent
The aim of this paper is to exhibit a self-contained and very intuitive proof of the decomposition of locally unital MTL-chains as an ordinal sum of archimedean MTL-chains in terms of poset products given in \cite{CZ2017}, by employing extensions of finite chains in the category of semihoops. Additionally, we propose a construction called Generalized Chain Product and we prove that such construction turns out to be a suitable generalization of ordinal sums for locally unital MTL-chains.
\\

\noindent
The first section is devoted to introduce all the concepts required to read this work. The second section is intended to give a characterization of those extensions of finite totally ordered semihoops which are isomorphic to an ordinal sum. In the third section we show that every finite locally unital MTL-chain can be decomposed as the ordinal sum of archimedean MTL-chains. Finally, in the last section, we introduce the Generalized Chain Products and we prove that the ordinal sum construction for finite totally ordered semihoops is indeed a particular case of these.

\section{Preliminaries}

A \emph{semihoop}\footnote{Also called basic semihoop in \cite{V2010}.} is an algebra $\textbf{A}=(A,\cdot, \rightarrow, \wedge, \vee, 1)$ of type $(2,2,2,2,0)$ such that $(A,\wedge,\vee)$ is lattice with $1$ as greatest element, $(A,\cdot,1)$ is a commutative monoid and for every $x,y,z\in A$ the following conditions hold:
\\

\begin{tabular}{cccccc}
\emph{(residuation)} & & & &  & $xy\leq z\; \text{if and only if}\; x\leq y\rightarrow z$
\\
\emph{(prelinearity)} & & & & & $(x\rightarrow y)\vee (y\rightarrow x) = 1$
\end{tabular}
\\

\noindent
Equivalently, a semihoop is an integral, commutative and prelinear residuated lattice. We write $\SH$ for the algebraic category of semihoops. A semihoop $\textbf{A}$ is \emph{bounded} if $(A,\wedge,\vee, 1)$ has a least element $0$. An \emph{MTL-algebra} is a bounded semihoop, hence, MTL-algebras are prelinear integral bounded commutative residuated lattices, as usually defined \cite{EG2001,NEG2005} and semihoops are basically ``MTL-algebras without zero''. An MTL-algebra (or semihoop) $\textbf{A}$ is an \emph{MTL-chain} (SH-chain) if its semihoop reduct is totally ordered. It is a well known fact, that the class of MTL-algebras is a variety. We write $\MTL$ for the category of MTL-algebras and MTL-homomorphisms. A totally ordered MTL-algebra is archimedean if for every $x \leq y < 1$, there exists $n \in \mathbb{N}$ such that $y^n \leq x$. A submultiplicative monoid $F$ of $M$ is called a filter if is an up-set with respect to the order of $M$. For every $x\in F$, we write $\langle x \rangle$ for the filter generated by $x$; i.e.,
\[\langle x \rangle = \{a \in M \ | \ x^n \leq a \; \text{for some } n \in \mathbb{N}\}.\]
\noindent
For any filter $F$ of $M$, we can define the relation $\sim_{F}$, on $M$ by $a\sim_{F} b$ if and only if $a\rightarrow b\in F$ and $b\rightarrow a\in F$. It follows that $\sim_{F}$ is indeed a congruence on $M$. On MTL-algebras there exists a well known correspondence between filters and congruences (c.f. \cite{NEG2005}) so we write the quotient $M/\sim_{F}$ by $M/F$. For every $a\in M$, we write $[a]_{F}$ for the equivalence class of $a$ in $M/F$. If there is no ambiguity, we simply write $[a]$. A filter $F$ of $M$ is \emph{prime} if $0\notin F$ and $x\vee y \in F$ entails $x\in F$ or $y\in F$, for every $x,y\in M$.
\\

\noindent
Let $\textbf{I}=(I,\leq)$ be a totally ordered set and $\mathcal{F}=\{\textbf{A}_{i}\}_{i\in I}$ a family of semihoops. Let us assume that the members of $\mathcal{F}$ share (up to isomorphism) the same neutral element; i.e., for every $i\neq j$, $A_{i}\cap A_{j}=\{1\}$. The \emph{ordinal sum} of the family $\mathcal{F}$, is the structure $\bigoplus_{i\in I} A_{i}$ whose universe is $\bigcup_{i\in I} A_{i}$ and whose operations are defined as:
\begin{displaymath}
 x\cdot y= \left\{ \begin{array}{lcl}
             x\cdot_{i}y, & \mbox{if} & x,y\in A_{i}
             \\ y, & \mbox{if} & x\in A_{i},\; \mbox{and}\; y\in A_{j}-\{1\},\; \mbox{with}\; i>j,
             \\ x, & \mbox{if} & x\in A_{i}-\{1\},\; \mbox{and}\; y\in A_{j},\; \mbox{with}\; i<j.
             \end{array}
   \right.
\end{displaymath}

\begin{displaymath}
 x\rightarrow y= \left\{ \begin{array}{lcl}
             x\rightarrow_{i}y, & \mbox{if} & x,y\in A_{i}
             \\ y, & \mbox{if} & x\in A_{i},\; \mbox{and}\; y\in A_{j},\; \mbox{with}\; i>j,
             \\ 1, & \mbox{if} & x\in A_{i}-\{1\},\; \mbox{and}\; y\in A_{j},\; \mbox{with}\; i<j.
             \end{array}
   \right.
\end{displaymath}
\noindent
where the subindex $i$ denotes the application of operations in $A_{i}$.
\\

\noindent
Furthermore, if $\textbf{I}$ has a minimum $\perp$, $A_{i}$ is a totally ordered semihoop for every $i\in \textbf{I}$ and $A_{\perp}$ is bounded then $\bigoplus_{i\in I} A_{i}$ becomes an MTL-chain. In order to clarify notation, we will use the symbol $\boxplus$ to denote the usual linear sum of lattices (as defined in Section 1.24 of \cite{DP}).
\\

\noindent
Let $M$ be an MTL-algebra. Write $0,1$ for the trivial idempotents, and $\Id(M)$ for the set of all idempotent elements of $M$ and $\nId(M)$ for $\Id(M)-\{0\}$. We say that $e\in \nId(M)$ is a \emph{local unit} if $xe=x$, for all $x\leq e$. Clearly $1$ is a local unit. If $M$ is archimedean, $1$ is in fact the only local unit of $M$. Notice that there may be idempotents that are not local units. As an example, one may consider the MTL-algebra $\textbf{A}$ whose underlying set is the totally ordered set $A=\{0,x,e,1\}$ and the product is determined by the following table:

\begin{table}[h]
\begin{center}
{\small $\begin{tabular}{c|c|c|c|c|}
$\cdot$ & 1 & e & x & 0
\\
\hline
1 & 1 & e & x & 0
\\
\hline
e & e & e & 0 & 0
\\
\hline
x & x & 0 & 0 & 0
\\
\hline
0 & 0 & 0 & 0 & 0
\\
\hline
\end{tabular}$
}
\end{center}
\caption{MTL-algebra without local units.}
\label{Table1}
\end{table}
\noindent
Let us to consider the following quasi-identity in the language of MTL-algebras (or semihoops):
\begin{itemize}
\item[(LU)] If $e^{2}=e$ and $e\vee x=e$ then $ex=x$.
\end{itemize}

\noindent
In \cite{CZ2017} the corresponding quasivariety resulting from adding $(LU)$ to the theory of MTL-algebras was called \emph{locally unital} MTL-algebras (luMTL, for short). In this this paper, locally unital MTL-algebras will be denoted by $\QMTL$. Observe that every archimedean chain is a member of $\QMTL$.

\section{Another perspective on finite luMTL-chains}
\label{Another perspective for finite luMTL-chains}

We start this section with the following observation: Let $M$ be a semihoop and $F$ be a filter of $M$. A straightforward verification shows that $F$ is a subalgebra of $M$, which basically sais that in $\SH$, filters always belong to the same category of the original algebra. Such situation does not hold for MTL-algebras, since the only filter of $M$ which is a subalgebra is $M$ itself. Hence, taking advantage of the fact that MTL-algebras are ``semihoops with zero", along this section we will place our selves in the context of semihoops, rather than MTL-algebras.
\\

\begin{rem}\label{Semihoops are semiabelian}
The proof of \cite{J2004}, showing that the category of Heyting semilattices is semi-abelian, also works to show that the variety of integral commutative residuated lattices is semi-abelian. In particular, we can conclude that $\SH$ is semi-abelian.
\end{rem}

\noindent
Let $K,C\in \SH$ with $C$ a chain. By Remark \ref{Semihoops are semiabelian}, $\SH$ is semi-abelian, so \emph{an extension $E$ of $K$ by $C$} is simply a split short exact sequence in $\SH$ of the form
\begin{displaymath}
\xymatrix{
\textbf{0} \ar[r] & K \ar@{>->}[r]^-{k} & E \ar@{>>}[r]_-{p} & C \ar@/_/[l]_-{s} \ar[r] & \textbf{0}
}
\end{displaymath}
\noindent
Here $\textbf{0}=\{1\}$ is the zero object of the category. An interesting problem is to classify all the extensions of this form. Nevertheless, in this paper we give a solution for a more concrete problem. The aim of this section is presenting a characterization of the extensions of $C$ by $K$ which are isomorphic to the ordinal sum of $C$ with $K$.
\\

\begin{lem}\label{Classification of some split extensions in SH}
Let $K,E,C \in \SH$ where $C$ is a chain. Consider the following short exact sequence:
\begin{displaymath}
\xymatrix{
\textbf{0} \ar[r] & K \ar@{>->}[r]^-{k} & E \ar@{>>}[r]^-{p} & C \ar[r] & \textbf{0}
}
\end{displaymath}
Then, $E\cong (C-\{1\}) \boxplus K$ if and only if the sequence is an extension $E$ of $K$ by $C$, such that:
\begin{itemize}
\item[i)] $E=k(K)\cup s(C)$,
\item[ii)] $k(K)\cap s(C)=\textbf{0}$.
\end{itemize}
\end{lem}

\begin{proof}
Let us assume that the short exact sequence of above is an extension $E$ of $K$ by $C$ satisfying $(i)$ and $(ii)$. Since $s,k$ are morphisms of semihoops, $s(C)$ and $k(K)$ are subobjects of $E$ isomorphic to $K$ and $C$, respectively. So let us simply write $K$ for $k(K)$ and $C$ for $s(C)$. Let $e\in E$. By $(i)$, $e\in K$ or $e\in C$. If $e\in K\cap C$, then, by $(ii)$, $e=1$. If $e\in K$ and $f\in E$ is such that $e\leq f$, then, from $1=p(e)\leq p(f)$ we get that $p(f)=1$. Hence $f\in K$. In consequence, $K$ is a filter in $E$, and every $c\in C-\{1\}$ is below every $k\in K$. Thereby, as a lattice, $E \cong (C-\{1\})\boxplus K$. Notice that, both $K$ and $C$ are closed by the product and the residuum. So we just have to calculate $ck$, $c\rightarrow k$ and $k\rightarrow c$ for $c\in C$ and $k\in K$. Observe that we can assume $c\neq 1$. From the calculation
\[ p(ck)=p(c)p(k)=p(c)\cdot 1 =p(c)\]
we can conclude that $ck=c$ for every $c\in C$ and $k\in K$. Taking into account the order of $E$, thus $ck=c\wedge k$. On the other hand, $p(k\rightarrow c)=1\rightarrow p(c)=p(c)$ so $k\rightarrow c=c$; $p(c\rightarrow k)=p(c)\rightarrow 1=1$, hence $c\rightarrow k\in K$. Finally, since $c\wedge k'=ck'\leq c \leq k$ for every $k'\in K$, then $c\rightarrow k=1$.
\\

\noindent
We have proved that the binary operations on $E\cong (C-\{1\})\boxplus K$ are given by:
\begin{displaymath}
kc= \left\{ \begin{array}{lll}
             k\cdot_{K} c, & k,c\in K
             \\ k\cdot_{C} c, & k,c\in C
             \\ k\wedge c,   & \mbox{otherwise}
             \end{array}
   \right.
\end{displaymath}

\begin{displaymath}
k\rightarrow c= \left\{ \begin{array}{lll}
             k\rightarrow_{K} c, & k,c\in K
             \\ k\rightarrow_{C} c, & k,c\in C
             \\ c,   & k\in K, c\in C
             \\ 1,   & \mbox{otherwise}
             \end{array}
   \right.
\end{displaymath}
\noindent
That is to say, $E\cong K\oplus C$, the ordinal sum of semihoops. The converse follows directly from the definition of ordinal sum of semihoops.
\end{proof}

\noindent
To conclude this section we would like to remark that in $\SH$ there are extensions that are not of the kind presented in Lemma \ref{Classification of some split extensions in SH}. In order to exhibit one, consider the semihoop $A=\{0,x,e,1\}$ whose product is defined in Table \ref{Table1}. Let $F={\uparrow}e$ and $\textbf{2}$ be the chain of two elements. It is clear that $A/F\cong \textbf{2}$ and that $s:\textbf{2}\rightarrow A$, defined by $s(1)=1_{A}$ and $s(0)=0_{A}$ is a section of the quotient map, so $\textbf{0}\rightarrow F\rightarrow  A \rightarrow \textbf{2}\rightarrow \textbf{0}$ is a short split exact sequence. Nevertheless, $x\notin s(\textbf{2})\cup k(F)$ so $A\ncong \textbf{2}\oplus F$.

\section{Dealing with finite chains in luMTL}

In \cite{B2004}, Busaniche proved that every BL-chain can be decomposed as an ordinal sum of totally ordered Wajsberg hoops. In \cite{CZ2017}, it was observed that, as a consequence of the divisibility condition for BL-algebras\footnote{For every $x,y$ the equation $x(x\rightarrow y)=x\wedge y$ holds.}, it follows that every BL-algebra is a locally unital MTL-algebra. Moreover, it can be proved that every totally ordered Wajsberg hoop is in fact an archimedean MTL-algebra. The aim of this section is to generalize the finite case of Busaniche's result to finite locally unital MTL-chains. That is to say, to give an elementary proof of the fact that every finite locally unital MTL-chain decomposes as an ordinal sum of archimedean MTL-chains (see \cite{CZ2017} for another proof).
\\

\noindent
Let $M$ be a finite MTL-chain (or SH-chain). The following useful characterization was proved in \cite{CZ2017}.

\begin{prop}\label{Characterization of finite archimedean MTL-chains}
Let $M$ be a finite MTL-chain. The following are equivalent:
\item[i)] $M$ is archimedean,
\item[ii)] $M$ is simple,
\item[iii)] $\Id(M)=\{0,1\}$.
\end{prop}
\noindent
Write $\langle X \rangle$ for the filter generated by $X\subseteq M$ and $\LU(M)$ for the set of local units of $M$. Note that in general $\LU(M)\subseteq \nId(M)$. Furthermore, if $M\in \QMTL$ then $\LU(M)= \nId(M)$.

\begin{prop}[Corollary 4, \cite{CZ2017}]
	\label{Equivalence filters and idempotents}
In any finite MTL-algebra $M$, the following are equivalent:
\begin{enumerate}
	\item[i)] $a \in \mathcal{I}(M)$,
	\item[ii)] $\langle a \rangle = \ {\uparrow} a$.
\end{enumerate}
\end{prop}
\noindent
Note that if $M$ is finite and $F$ is a proper filter of $M$, there exists a unique $e\in \nId(M)$ such that $F={\uparrow} e$. We write $eM$ for the set $eM=\{ex\mid x\in M\}$.

\begin{rem}\label{Observation 1}
Let $M$ be a finite MTL-chain and $e\in \LU(M)$. Then the extension
\begin{displaymath}
\xymatrix{
\textbf{0} \ar[r] & {\uparrow} e \ar[r] & M \ar[r] & M/{\uparrow} e \ar[r] & \textbf{0}
}
\end{displaymath}
is an extension of the type of Lemma \ref{Classification of some split extensions in SH} and $M/{\uparrow} e \cong eM$ (which is in fact true for any $e\in \LU(M)$). Hence $M\cong {\uparrow} e \oplus eM$.
\end{rem}

\begin{lem}\label{Every QMTL finite chain is an ordinal sum of archimedian chains}
Every finite locally unital MTL-chain is an ordinal sum of archimedean chains.
\end{lem}
\begin{proof}
Let $M$ be a finite locally unital MTL-chain. It follows that $\LU(M)$ is also a finite chain. In order to prove the statement, we proceed by induction over the amount of local units of $M$. If $|\LU(M)|=1$, then by Proposition \ref{Characterization of finite archimedean MTL-chains}, $M$ is archimedean. Let us assume that $|\LU(M)|>1$ and let us order its elements in an increasing way:
\[e_{1}<...<e_{n}=1\]
\noindent
Let us suppose that any locally unital MTL-chain $N$ with $|\LU(N)|=n-1$, is an ordinal sum of $n-1$ archimedean chains:
\[N\cong C_{1}\oplus ...\oplus C_{n-1}\]
\noindent
Since $e_{n-1}\in \LU(M)$, then, by Remark \ref{Observation 1}, the short exact sequence
\begin{displaymath}
\xymatrix{
\textbf{0} \ar[r] & {\uparrow} e_{n-1} \ar[r] & M \ar[r]_-{p} & e_{n-1}M \ar@/_/[l]_-{s} \ar[r] & \textbf{0}
}
\end{displaymath}
\noindent
is a split extension, and hence, by Lemma \ref{Classification of some split extensions in SH}, $M\cong {\uparrow} e_{n-1} \oplus e_{n-1}M$. Furthermore, ${\uparrow} e_{n-1}$ is archimedean by Proposition \ref{Characterization of finite archimedean MTL-chains}, and $\LU(e_{n-1}M)=n-1$. By inductive hypothesis we can conclude that $e_{n-1}M\cong C_{1}\oplus ...\oplus C_{n-1}$. An easy verification shows that
\[M\cong {\uparrow} e_{n-1}\oplus e_{n-1}({\uparrow} e_{n-2})\oplus ... \oplus e_{2}({\uparrow} e_{1})\oplus e_{1}M\]
This concludes the proof.
\end{proof}

\section{Generalized Chain Products}\label{Generalized Chain Products}

In this section we introduce the concept of Generalized Chain Products of MTL-chains and we show its intimate relation with extensions by chains of totally ordered semihoops. Furthermore, we show that generalized chain products are in fact, a generalization of ordinal sums of locally unital MTL-chains.
\\

\noindent
Let $F$ and $A$ be two finite MTL-chains. Since $A$ is finite, we can order its elements in an increasing way; let us say, $0_{A}<x_{1}<...<x_{n-1}<1_{A}$. Let $\mathcal{C}_{(F,A)}=\{C_{j}\mid j\in A\}$ a family of sets such that:
\begin{enumerate}
\item For every $i\in A$, $C_{i}$ is a finite chain,
\item For $i=1_{A}$, $C_{1_{A}}=F$, and
\item If $i \neq j$, with $i, j \in A$, then $C_{i} \cap C_{j} = \emptyset$.
\end{enumerate}
\noindent
Let $E:= \bigcup_{i\in A}C_{i}$ with the order induced by the ordinal sum of lattices; that is, as a poset, $E = {\mathrm{\boxplus}}_{i\in A} C_{i}$. Hence $E$ is a lattice. Observe that, since each $C_{i}$ is a finite lattice, then it has a minimum element. Let us write $0_{i}$ for it.
\\

\noindent
Now, take a family of functions $\textbf{M}_{(F,A)}=\{\mu_{ij}:C_{i}\times C_{j}\rightarrow C_{i \cdot j}\mid i,j\in A\}$ (where $i \cdot j$ denotes the product in $A$), such that:
\begin{itemize}
\item[i)] For every $i,j\in A$, $\mu_{ij}$ is monotone in each coordinate,
\item[ii)] For every $k_{i}\in C_{i}$ and $k_{j}\in C_{j}$, $\mu_{ij}(0_{i},k_{j})=\mu_{ij}(k_{i},0_{j})=0_{ij}$,
\item[iii)] $\textbf{M}_{(F,A)}$ is \emph{jointly associative}; that is, the following diagram
\begin{displaymath}
\xymatrix{
C_{i}\times C_{j}\times C_{k} \ar[r]^-{id_{i}\times \mu_{jk}} \ar[d]_-{\mu_{ij}\times id_{k}} & C_{i}\times C_{j \cdot k} \ar[d]^-{\mu_{i \cdot (j \cdot k)}}
\\
C_{i \cdot j}\times C_{k} \ar[r]_-{\mu_{(i \cdot j)\cdot k}} & C_{i\cdot (j\cdot k)} = C_{(i \cdot j) \cdot k}
}
\end{displaymath}
\noindent
commutes, for every $i,j,k\in A$.
\item[iv)] $\textbf{M}_{(F,A)}$ is \emph{jointly commutative}; that is, the following diagram
\begin{displaymath}
\xymatrix{
C_{i}\times C_{j} \ar[r]^-{\mu_{ij}} \ar[d]_-{\tau} & C_{i \cdot j} \ar@{=}[d]\\
C_{j}\times C_{i} \ar[r]_-{\mu_{ji}} & C_{j \cdot i}
}
\end{displaymath}
\noindent
commutes, for every $i,j,k\in A$.
\item[v)] $\textbf{M}_{(F,A)}$ has a \emph{global unit}; that is, the following diagram
\begin{displaymath}
\xymatrix{
\textbf{1}\times C_{i} \ar[r] \ar[d]_-{1_{F} \times id_{i}} & C_{i} \\
C_{1}\times C_{i} \ar[ur]_-{\mu_{1i}} &
}
\end{displaymath}
\noindent
commutes, for every $i \in A$. Here $\textbf{1}$ is the singleton chain and $1$ is the unit of $A$.
\end{itemize}
\noindent
The latter presentation lead us to introduce the following concept.

\begin{defi}\label{Generalized chain product}
Let $F$ and $A$ be two finite MTL-chains. A \emph{generalized chain product} (induced by $F$ and $A$) is a pair $(\mathcal{C}_{(F,A)},\textbf{M}_{(F,A)})$ such that:
\begin{itemize}
\item[1.] The collection $\mathcal{C}_{(F,A)}=\{C_{j}\mid j\in A\}$ satisfies $1.$, $2.$ and $3$,
\item[2.] The set $\textbf{M}_{(F,A)}=\{\mu_{ij}:C_{i}\times C_{j}\rightarrow C_{i \cdot j}\mid i,j\in A\}$ is a family of functions satisfying $(i)$ to $(v)$.
\end{itemize}

\end{defi}

\begin{lem}\label{Generalized chain products arise extensions}
Let $F$ and $A$ be two finite MTL-chains. Then, every generalized chain product $(\mathcal{C}_{(F,A)},\textbf{M}_{(F,A)})$ defines an extension of $A$ by $F$.
\end{lem}
\begin{proof}
Take $E(\mathcal{C}_{(F,A)},\textbf{M}_{(F,A)}):= \bigcup \mathcal{C}$ with the mentioned order. Endow $\textbf{E}=E(\mathcal{C}_{(F,A)},\textbf{M}_{(F,A)})$ with the following binary operation:
\[\mu: \textbf{E}\times \textbf{E} \rightarrow \textbf{E} \]
defined by $\mu(e,f)=\mu_{ij}(e,f)$, if $e\in C_{i}$ and $f\in C_{j}$. Observe that conditions $(i)$ to $(v)$ of $\textbf{M}$ guarantee that $(\textbf{E}, \vee, \wedge, \mu, 1_{F})$ is an MTL-algebra. Furthermore, $F = C_{1}$ is a filter of $\textbf{E}$. Let us calculate $\textbf{E}/F$. For $e,f\in \textbf{E}$, $e\sim_{F} f$ if and only if there is a $g\in F$ such that $ge\leq f$ and $gf\leq e$, so it is clear that $F=[1]$. Now, let us assume $e,f\in C_{i}$ for some $i\neq 1_{A}$. From $(ii)$, $\mu_{1_{A}i}(0_{1_{A}},e)=\mu_{1_{A}i}(0_{1_{A}},f)=0_{i}$ so $e\sim_{F} f$. If $e\in C_{i}$ and $f\in C_{j}$ for some $i\neq j$, then, since $C_{i}\cap C_{j}=\emptyset$, thus for every $g\in F=C_{1_{A}}$ we have that $\mu_{1_{A}i}(g,e)\in C_{i}$ and $\mu_{1_{A}j}\in C_{j}$; that is, $e\nsim_{F}f$. Hence $C_{i}=[e]_{F}$ for any $e\in C_{i}$ and $i\in A$. So we have a bijection $\textbf{E}/F\rightarrow A$ which clearly is an MTL-morphism. Moreover, the map $s:A\rightarrow \textbf{E}$ defined by
\begin{displaymath}
s(i)= \left\{ \begin{array}{ll}
             1_{F}, & i=1_{A}
             \\ 0_{i}, & \mbox{otherwise}
             \end{array}
   \right.
\end{displaymath}

\noindent
is a section of $p:\textbf{E}\rightarrow A\cong \textbf{E}/F$. Hence, we get an split short exact sequence in $\SH$:
\begin{displaymath}
\xymatrix{
\textbf{0} \ar[r] & F \ar[r]^-{j} & \textbf{E} \ar@{->>}[r]_-{p} & A \ar@/_/[l]_-{s} \ar[r] & \textbf{0}
}
\end{displaymath}
\noindent
where $j$ is the inclusion.
\end{proof}
\noindent
The following result is a partial converse of Lemma \ref{Generalized chain products arise extensions}.

\begin{lem}\label{Extensions arise generalized chain products}
Let
\begin{displaymath}
\xymatrix{
\textbf{0} \ar[r] & F \ar[r]^-{j} & \textbf{E} \ar[r]^-{p} & A \ar[r] & \textbf{0}
}
\end{displaymath}
\noindent
be a split exact sequence in $\SH$ such that:
\begin{itemize}
\item[i)] $j(F)$ is a convex subset of $E$,
\item[ii)] $E,F$ and $A$ are finite chains.
\end{itemize}
\noindent
Then, there exists a generalized chain product $(\mathcal{C}_{(F,A)},\textbf{M}_{(F,A)})$ such that $\textbf{E}\cong E(\mathcal{C}_{(F,A)},\textbf{M}_{(F,A)})$.
\end{lem}
\begin{proof}
Observe that, since $j(F)\cong F$ is a convex submonoid of a finite partially ordered commutative integral monoid\footnote{Also called negative ordered monoid in \cite{V2010}.}, $F$ is a filter of $E$ and hence, there exists a unique idempotent $e_{F}\in E$ such that $F={\uparrow} e_{F}$. In the future we shall simply write $e$ for $e_{F}$. Notice that there is an isomorphism between $A$ and $eE$. By definition  of extension, we know that $A$ is just $E/F$, so we can define a map $f:E/F \rightarrow eE$ by $f([a])=ea$ which is clearly bijective. From this fact we can conclude that $[a]=\{x\in E \mid ex=ea\}$. Now, since each equivalence class is a convex set of $E$ (because each class is in particular a lattice congruence class), we have a decomposition of $E$ as a lattice as $E\cong {\mathrm{\boxplus}}_{[x]\in A} C_{[x]}$, where the $C_{[x]}$'s form a partition of $E$ by convex chains. Clearly, this corresponds to the quotient map. Now take $\mathcal{C}=\{[a]\mid [a]\in A\}$ and define $\mu_{[a][b]}:[a]\times [b]\rightarrow [ab]$ by $\mu_{[a][b]}(\alpha,\beta)=\alpha\beta$; that is, the restriction of the product to the respective $C_{[x]}$'s. It is easy to check that they form a family $\textbf{M}$ satisfying conditions $(i)$ to $(v)$ of the begining of this section.
\end{proof}
\noindent
To conclude this paper, we show how Lemma \ref{Extensions arise generalized chain products} justifies the use of the term ``generalized chain product'' in Definition \ref{Generalized chain product}, in the sense that regular sum products (ordinal sums) in locally unital MTL-chains are a very particular sort of generalized chain products.

\begin{coro}\label{Ordinal sums are generalized chain products}
Ordinal sums of finite locally unital MTL-chains are generalized chain products.
\end{coro}
\begin{proof}
Let $\{M_{i}\}_{i\in I}$ be a family of finite locally unital MTL-chains, where $I$ is a finite totally ordered set with top $\top$. Take $\{M'_{i}\}_{i\in I}$ with $M'_\top = M_\top$ and $M'_i = M_i - \{ 1_{M_i} \}$ for $i \neq \top$. Note that each $M'_i$ is a semigroup. Take $i\cdot j = i \wedge j$, for $i, j \in I$. This makes of $I$ a finite MTL-algebra (in fact, a Heyting algebra). By regarding $\mathcal{C}_{(M'_{T},I)}=\{M'_{i}\}_{i\in I}$ and $\textbf{M}_{(M'_{T},I)}=\{\mu_{ij}:M'_{i}\times M'_{j}\rightarrow M'_{i \wedge j} \mid i,j\in I\}$, defined as $\mu_{ij}(a, b)= a \cdot_{i} b$, if $i = j$; $\mu_{ij}(a, b) = a$, if $i < j$; and $\mu_{ij}(a, b)= b$, if $j < i$; from Lemmas \ref{Every QMTL finite chain is an ordinal sum of archimedian chains} and \ref{Generalized chain products arise extensions} the result follows.
\end{proof}



\end{document}